\documentclass[twoside,twocolumn]{article}

\usepackage{tabulary,graphicx,times,caption,fancyhdr,amsfonts,amssymb,amsbsy,latexsym,amsmath}
\usepackage[utf8]{inputenc}
\usepackage{url,multirow,morefloats,floatflt,cancel,tfrupee,textcomp,colortbl,xcolor,pifont}
\usepackage[nointegrals]{wasysym}
\usepackage{amsthm}
\usepackage{float}
\usepackage{bm}
\usepackage{tikz}
\usepackage{mathdots}
\usepackage{yhmath}
\usepackage{siunitx}
\usepackage{array}
\usepackage{gensymb}
\usepackage{tabularx}
\usepackage{booktabs}
\usepackage{subfigure}

\usepackage{lipsum}
\usepackage{mathtools}
\usepackage{cuted}
\usepackage{placeins}

\makeatletter

\usepackage{ifxetex}
\ifxetex\else
  \usepackage{dblfloatfix}
\fi

\@ifundefined{subparagraph}{
\def\subparagraph{\@startsection{paragraph}{5}{2\parindent}{0ex plus 0.1ex minus 0.1ex}%
{0ex}{\normalfont\small\itshape}}%
}{}

\def\URL#1#2{\@ifundefined{href}{#2}{\href{#1}{#2}}}

\def\UrlOrds{\do\*\do\-\do\~\do\'\do\"\do\-}%
\g@addto@macro{\UrlBreaks}{\UrlOrds}

\makeatother

\usepackage[paperheight=11in,paperwidth=8.3in,margin=2.5cm,headsep=.7cm,top=2.5cm]{geometry}
\usepackage[T1]{fontenc}

\renewenvironment{abstract}
	{\trivlist\item[]\leftskip0pt\par\vskip4pt\noindent
  	\textbf{\abstractname}\mbox{\null}\\}
	{\par\noindent\endtrivlist}

\def\keywords#1{\par\medskip\par\noindent\textbf{Keywords}: #1\par}

 \date{} \emergencystretch 8pt

\makeatletter
\def\author#1{\gdef\@author{\hskip-\tabcolsep%
	\parbox{\textwidth}{\raggedright\bfseries#1\\[1pc]}}}
\def\address[#1]#2{\g@addto@macro\@author{\\\hskip-\tabcolsep\parbox{\textwidth}{\raggedright%
	\normalsize\normalfont\textsuperscript{#1}#2}}}
\let\addresslink\textsuperscript
\def\correspondence#1{\g@addto@macro\@author{\\\hskip-\tabcolsep\parbox{\textwidth}{\raggedright%
	\vspace*{10pt}\normalsize\normalfont~\\#1~\\[12pt]}}}
\def\email#1{\g@addto@macro\@author{\\\hskip-\tabcolsep\parbox{\textwidth}{\raggedright%
	\normalsize\normalfont Emails: #1}}}

\def\title#1{\gdef\@title{\vspace*{-30pt}%
	\raggedright\textbf{\@journaltitle}~\\%
  \raggedright\bfseries\ifx\@articleType\@empty\vspace*{20pt}\else%
  \vspace*{20pt}\@articleType\vspace*{20pt}\\\fi#1}}
\let\@journaltitle\@empty \def\journaltitle#1{\gdef\@journaltitle{{\normalfont\itshape#1}}}
\let\@articleType\@empty \def\articletype#1{\gdef\@articleType{{\normalfont\itshape#1}}}

\let\@runningHead\@empty \def\RunningHead#1{\gdef\@runningHead{{\normalfont #1}}}

\usepackage{fancyhdr}
\fancypagestyle{headings}{\fancyhf{}
  \fancyhead[R]{\itshape\@runningHead}
  \fancyfoot[C]{\thepage}}
\makeatother

\usepackage[%
	numbers,sort&compress%
  ]{natbib}

\theoremstyle{plain}
\newtheorem{theorem}{Theorem}

\newtheorem{corollary}[theorem]{Corollary}
\newtheorem{proposition}[theorem]{Proposition}
\newcommand{\RNum}[1]{\uppercase\expandafter{\romannumeral #1\relax}}
\theoremstyle{definition}
\newtheorem{definition}[theorem]{Definition}

\theoremstyle{remark}
\newtheorem{remark}{Remark}

\usepackage{float,xcolor}

\begin{document}

\title{The Stability of $\alpha-$ Harmonic Maps with Physical Applications}

\author{%
		Seyed Mehdi Kazemi Torbaghan\addresslink{1}
		and
  	 Keyvan Salehi\addresslink{2} 
  }
		
\address[1]{ Department of Mathematics, Faculty of Basic Sciences, University of Bojnord,  Bojnord,  Iran;}
\address[2]{Central of Theoretical Chemistry And Physics (CTCP), Massey University, Auckland, New Zealand;}

\correspondence{Correspondence should be addressed to 
  Seyed Mehdi Kazemi Torbaghan: m.kazemi@ub.ac.ir}

\email{Keyvan Salehi(k.salehi@massey.ac.nz)}

\RunningHead{ The Stability of $\alpha-$ Harmonic Maps with Physical Applications}

\maketitle 

\begin{abstract} 
The first result in this study is a non-existence theorem for $\alpha-$harmonic mappings. Additionally,   a direct connection between the $\alpha-$ harmonic and harmonic maps is made possible via conformal deformation. Second,  the instability of non-constant $\alpha$-harmonic maps is investigated with regard to the target manifold's Ricci curvature requirements.  Next,  the concept of $\alpha-$stable manifolds and their physical applications are explored.  Finally,  it is investigated the $\alpha-$stability of compact Riemannian manifolds that admit a non-isometric conformal vector field as well as the Einstein Riemannian manifolds under certain assumption on   the smallest positive  eigenvalue of its Laplacian operator on functions.

%
%
\end{abstract}
\section{Introduction}
\noindent

Mathematical physicist have devoted a great deal of study to the theory of harmonic mappings between Riemannian manifolds,  \cite{hhhh,s1, D1}.  Let $\psi:(M,g)\longrightarrow (N,h)$ be a smooth map from a compact Riemannian manifold $(M,g)$  to an arbitrary Riemannian manifold $(N,h)$.  The energy functional of $\psi$ is defined as follows
\begin{equation}
E(\psi)=\frac{1}{2} \int _{M} \mid d\psi \mid ^2 dV _{g},
\end{equation}
where 
$dV_{g}$
is the volume element of $(M,g)$.  A smooth map $\psi$ is said to be harmonic if $E$ is stationary at $\psi$.  This is equivalent to saying that $\psi$ satisfies the Euler-Lagrange equation corresponding to $E$ given by 
\begin{equation}\label{123er}
\tau (\psi):= Tr_{g } \,\nabla d \psi =0,
\end{equation}
here $\nabla $ is the induced connection on the pull-back bundle $\psi ^{-1}TN$. 
The section $\tau(\psi)\in \Gamma (\psi ^{-1}TN)$ described in \eqref{123er},  is known as the \textit{tension field} of $\psi$.
Eells and Sampson \cite{js} originally established the fundamental existence theorem for  harmonic maps under the premise that the sectional curvature of the target manifold  is non-positive.  However,   if the sectional curvature of the target manifold  is positive,   then the heat flow equation may blow up in finite time, \cite{cd}.  \\

  Recently,   several contributions about harmonic maps have been made by mathematical physicist,  \cite{a1,cd,D1,m1}. This is owing to the application of harmonic maps in numerous theories in mathematical physics, such as liquid crystal, ferromagnetic material, and superconductors.  For instance,  the solutions to the Landau-Lifshitz equation are equal to the harmonic map equation.
Furthermore,  given certain assumptions, it is possible to demonstrate that the Landau-Lifshitz equation has the same form as the heat flow for harmonic maps, \cite{h}. It is worth mentioning that Landau-Lifshitz first calculated the Landau-Lifshitz equation on phenomenological grounds in \cite{ll} and that it has played an important role in the understanding of non-equilibrium magnetism, \cite{cd}. \\

Harmonic maps from closed Riemannian surfaces and their variants are useful in both mathematics and physics as instruments for probing the geometry of a Riemannian manifold.  In physics,  they are  borderline examples for the Palais- Smale condition and so can not be reached directly using normal techniques.   For solving,  Sacks and Uhlenbeck in their pioneering paper \cite{542} in 1981,  introduced perturbed energy functional that satisfies the Palais- Smale condition,  and thus obtained so-called $\alpha$-harmonic maps as critical points of perturbed functional to approximate harmonic maps. They also demonstrated  that if $M$ is a compact Riemannian manifold, then any non-trivial class in $\pi_{2}(M)$ can be represented by a sum of smooth harmonic mappings,
$\psi_{j}: \mathbb{S}^{2}\longrightarrow M,  \, j=1,\cdots ,n$, for some positive integer $n$, \cite{542}. \\


 For $\alpha >1$, the \textit{$\alpha-$ energy functional} of the map $\psi$ is denoted by $E_{\alpha}(\psi)$ and defined as follows:
\begin{equation}
E_{\alpha}(\psi):=\int _{M}(1+\mid d\psi \mid ^{2})^{\alpha}dV _{g},
\end{equation}
where
$\mid d\psi \mid^{2}$
 denotes the Hilbert-Schmidt norm of the differential map $ d\psi \in \Gamma (T^{*}M\times \psi^{-1}TN)$ 
 with respect to $g$ and $h$. 
 This function can be seen as a perturbation of the energy functional $E$,  that is,  in contrast to $E$,  satisfying the Palais-Smale condition.  Moreover,  
$E_{\alpha}$ satisfies the
Morse theory and Ljusternik- Schnirelman theory, \cite{s1}.  
\\

Much research has been undertaken on $\alpha-$  harmonic maps.  For example, the authors of \cite{s2}  examined the convergence behavior of a sequence of $\alpha-$harmonic mappings $\psi_{\alpha} $ with $E_{\alpha}(\psi_{\alpha})<C$,  from a compact surface $(M,g)$ to a compact without boundary Riemannian manifold $(N,h)$.  It is worth noting that this sequence converges weakly to a harmonic map. In \cite{3s},  it was demonstrated that for a sequence of minimizing $\alpha$-harmonic mappings,  the necks converge to certain limit geodesics of limited length, implying that the  identity stays stable.  Furthermore,  it can be seen reference
\cite{7s} for an examination of energy-reducing sequences in homotopy classes. \cite{14s} for the energy identity of a sequence of $\alpha$-harmonic maps with an additional entropy-type condition and  \cite{s1},  for  the energy identity of  a sequence of $\alpha$-harmonic maps into  specific unique target manifolds.\\

	In terms of mathematical physics,
$\alpha-$harmonic maps
 are used to undertake essential work in the 
  theory of gauge fields,  as a generalization of the theory of classical electromagnetic fields,  that underpins the standard model of particle physics,  \cite{3s, 14s}.  For example,  the theories of $\alpha-$harmonic maps are employed as mathematical tools to demonstrate how soap films arrange themselves into form  that minimizing  their energy.   However,  these theories  had been tainted by the occurrence of points where energy appears to become infinitely concentrated.  Uhlenbeck studied  these points and  showed  that this is caused by a new bubble  splitting off the surface,   \cite{es}.  In 2019,  K. Uhlenbeck,  as the first woman,  won prestigious Abel prize for her prominent works on $ \alpha- $ harmonic maps from minimal surfaces.\\

  In this work,  we investigate the stability of $\alpha$-harmonic maps and $\alpha$-stability of Riemannian manifolds,  as well as their practical applications,  using the ideas of \cite{a1, a11, ch1999,c2018,H,hhhh, 542}.  In this regard,  the instability of non-constant $\alpha$-harmonic mappings with respect to the Ricci curvature criterion of the target manifold is studied.  Also,  the concept of $\alpha-$stable manifolds and their physical applications are investigated.  It is further probe the conditions under which  the  compact Riemannian manifolds admitting a non-isometric conformal vector field and Einstein Riemannian manifolds are $\alpha-$stable. Sections 3 and 4 summarize our significant findings. \\

This manuscript is arranged as follows.\\ Section 2 delves into the notions  of $\alpha-$ energy functional and $\alpha-$ harmonic maps,  and an explicit link between $\alpha-$ harmonic  and harmonic maps via conformal deformation is offered. 
In addition,  we look into the existence of 
$\alpha-$
 smooth map between  Riemannian
 manifolds.  On the other hand,  by reviewing the concepts of conformal and killing vector fields and their physical applications,  a 
 non-existence theorem for $\alpha$-harmonic maps is given. 
Section 3 investigates the second variation formula of the $\alpha-$energy functional,  as well as the notion of the stability of $\alpha$-harmonic maps  with regard to the Ricci curvature criterion of the target manifold. 
The concept of $\alpha-$stable manifolds and its practical applications are discussed in the final section.  In this regard,  the $\alpha-$instability of compact Riemannian manifolds admitting a non-isometric conformal vector field is then examined.   It is further probe the conditions under which any compact Riemannian manifolds admitting  killing vector fields is $\alpha-$stable.   Finally,  the stability of any  Einstein Riemannian manifold  is studied by making an assumption on the minimum eigenvalue of  its Laplacian operator on  functions.

%

\section{\large{ $\alpha$- harmonic maps}} 
\noindent

In this part, we  first investigate the concept of $\alpha-$ harmonic maps.  The specific relationship between $\alpha$- harmonic and harmonic maps via conformal deformation is then provided. Finally, non-constant $\alpha-$harmonic mappings  between Riemannian manifolds are examined.\\

Consider the smooth map $\psi: (M^{m}, g)\longrightarrow (N^{n}, h)$ between Riemannian manifolds. Throughout this study,  it is assumed that $ (M,g) $ is an $m-$dimensional compact,  orientable and without boundary Riemannian manifold.  Additionally, the Levi-Civita connections on $ M$ and $ N $ are represented by $ \nabla^{M} $ and  $ \nabla ^N $,  respectively.  In addition, the induced connection on the pullback bundle $ \psi^{-1}TN $ is denoted by $\nabla $ and defined as $ \nabla _{Y}V= \nabla ^{N} _{d\psi (Y)}V,$  for any smooth vector field $Y\in \chi(M)$ and section $ V \in \Gamma(\psi^{-1}TN )$.  Let $\alpha$ be a positive constant with a value greater than one. The $\alpha-$energy functional of $\psi $ is  defined by
\begin{equation}\label{34rtz}
E_{\alpha}(\psi):=\int _{M}(1+\mid d\psi \mid ^{2})^{\alpha}dV _{g}.
\end{equation}
The critical points of $E_{\alpha}$ are called $\alpha-$harmonic maps. The related Euler-Lagrange equation of the $\alpha-$energy functional,   $E_{\alpha}$, is given by Green's theorem as follows
 \begin{align}\label{h1}
\tau _{\alpha}(\psi)&:=2\alpha(1+\mid d\psi\mid^{2})^{\alpha-1}\tau(\psi)\nonumber \\&+d\psi(grad(2\alpha(1+\mid d\psi\mid^{2})^{\alpha-1})\nonumber \\&=0.
\end{align}
The section $\tau _{\alpha}(\psi)\in \Gamma(\psi ^{-1}TN)$ is said to be the $\alpha -$ tension field of $\psi$. \\
By \eqref{123er} and \eqref{h1}, we get the following theorem. 
\begin{theorem}
Let $\psi: (M, g)\longrightarrow (N, h)$ be a smooth map between 
Riemannian manifolds. Then, $\psi$ is $\alpha-$harmonic if and only if it has a vanishing $\alpha-$tension field.

\end{theorem}
\begin{definition}
Let 
$\psi:(M,g)\longrightarrow (N,h)$ be a smooth map between Riemannian manifolds.  It is called that 
$\varphi$ is non-degenerate if the induced tangent map 
$\psi_{\ast}=d\psi$
is non-degenerate,  i.e.  
$ker d\psi=\emptyset. $
\end{definition}
Now an explicit relation between 
$\alpha-$ harmonic maps and harmonic maps through conformal deformation is given. \\

\begin{proposition}\label{j23}
Let 
$\psi:(M^{m},g)\longrightarrow (N^{n},h)$ 
be a non-degenerate smooth map with
$m>2.$
Then,  $\psi$ is  $\alpha-$ harmonic  if and only if the map 
$\psi$ is harmonic with respect to the conformally related metric 
$\bar{g}$
defined by 
\begin{equation}
\bar{g}=\{(2\alpha)^{\dfrac{2}{m-2}}(1+\mid d\psi\mid^{2})^{\dfrac{2\alpha-2}{m-2}}\}g.
\end{equation}
\end{proposition}
\begin{proof}
Assume that
$\bar{g}$
be a Riemannian metric for some smooth positive function $\mu$ on $M$ conformally associated to 
$g$
by 
$\bar{g}=\mu^{2} g.$
Denote the tension fields of the smooth map $\psi$ with respect to 
$g$
and  
$\bar{g}$
by 
$\tau(\psi)$
and 
$\bar{\tau}(\psi)$, 
respectively.  By 
\eqref{123er},  it can be seen  that 
\begin{equation} \label{345rt}
\bar{\tau}(\psi)=\dfrac{1}{\mu ^{m}}\{\mu ^{m-2}\tau(\psi)+ d\psi (grad \mu ^{m-2})\}. 
\end{equation}
Setting 
\begin{equation}\label{r5}
\mu =(2\alpha)^{\dfrac{1}{m-2}}(1+\mid d\psi\mid^{2})^{\dfrac{\alpha-1}{m-2}}.
\end{equation}

By \eqref{h1}, \eqref{345rt} and \eqref{r5}, we get 
\begin{align}\label{345zh}
&\mu ^{m}\bar{\tau}(\psi)\nonumber \\&=\mu ^{m-2}\tau(\psi)+ d\psi (grad \mu ^{m-2})\nonumber \\ &=2\alpha(1+\mid d\psi\mid^{2})^{\alpha-1}
+ d\psi (grad(2\alpha(1+\mid d\psi\mid^{2})^{\alpha-1}) )
\nonumber \\&=\tau_{\alpha}(\psi).
\end{align}
The proposition \ref{j23} follows from 
\eqref{345zh}. 
\end{proof}
 Let $(M,g)$
and 
$(N,h)$
be compact Riemannian manifolds and $\mathcal{H}$
denote the  homotopy class of a smooth given map 
$\psi:(M^{m},g)\longrightarrow (N^{n},h)$. 
By using the methods of \cite{H},   the following theorem for the existence of $\alpha-$harmonic maps in 
$\mathcal{H}$
 can be obtained.
\begin{theorem}
Let
$\psi \in \mathcal{H}$ and $\psi:(M^{m},g)\longrightarrow (N^{n},h)$
be a harmonic map. Then, there is a smooth metric $\bar{g}$ on $M$ conformally equivalent to $g$ such that $\psi:(M,\bar{g})\longrightarrow (N,h)$ is $\alpha-$harmonic if $m>2\alpha$.
\end{theorem}
Now,  the concepts of conformal and killing vector fields are considered.  Then,  their physical applications are mentioned.  Ultimately,  by these concepts,  a non-existence theorem for $\alpha-$harmonic maps is given.  
\begin{definition}
A smooth vector field $X$ on a Riemannian manifold $(M,g)$ is called a conformal vector field  with  the potential function
$\lambda$   If  $\mathcal{L}_{X}g=2\lambda g$,  where $\mathcal{L}_{X}g$ is the Lie derivative of $g $ with respect to $X$.   A conformal vector field $X$  with the potential $\lambda$ is called killing If  $\lambda=0$ or,  equivalently,  the flow of $X$ consists of isometries of   $(M,g)$.  
\end{definition}
  Noting that any conformal vector field $X$ with potential function 
$\lambda $ satisfies 
\begin{equation}\label{20r}
g(\nabla _{Y}^{M}X,Z)+g(Y,\nabla _{Z}^{M}X)=\lambda g(Y,Z),
\end{equation}
for any $Y,Z\in \chi(M)$.\\
\begin{remark}
Conformal vector fields are key tools in physics for constructing various solutions that are then employed in the analysis of physical parameters in modified gravity theories (MGTs). By using correct conformal vector fields of $pp-$wave space-times in the $f(R)$ theory of gravity,  it is possible to discover that a very specific class of $pp-$waves known as plane fronted gravitational waves $(GWs)$ has a solution in the $f(R)$ theory of gravity,  \cite{gh1}.  Killing vector fields are also important in many domains of mathematics, including gravitation, quantum and teleparallel theory, and so on.
For example, in the realm of teleparallel theory, Sharif and Amir \cite{5tr} developed the teleparallel version of the Lie derivative for Killing vector fields and utilized those equations to find the teleparallel Killing vector fields in the Einstein manifold.  With this notion of teleparallel Lie derivative, an approach to exploring symmetries in teleparallel theory has been launched.
\end{remark}

Based on the above notations,  the non-existence of non-constant $\alpha-$harmonic maps is investigated in the following theorem.
\begin{theorem}\label{er3}
Let $\psi:(M,g)\rightarrow (N,h)$ be an $\alpha-$harmonic map between Riemannian manifolds.  Suppose that the Riemannian manifold $(N,h)$ admitting a conformal vector field $X$ with the potential function $\lambda>0$.  Then, 
$\psi$ is constant.  
\end{theorem}
\begin{proof}
Let $\{e_{i}\}$ be a normal orthonormal frame
at $z\in M$.  Setting 
\begin{equation}\label{21r}
\omega (Y):=h(X\circ \psi,  (1+\mid d\psi \mid ^{2})^{\alpha}d\psi(Y)). 
\end{equation}
By  \eqref{20r} and \eqref{21r},
 the divergence of $\omega$ can be obtained as follows 
\begin{align}\label{22r}
div(\omega)&=e_{i}(h(X\circ \psi,  (1+\mid d\psi \mid ^{2})^{\alpha}d\psi(e_{i})))\nonumber \\&=(1+\mid d\psi \mid ^{2})^{\alpha}h(\nabla _{e_{i}}X\circ \psi,d\psi(e_{i}) )
\nonumber \\&=\lambda\circ \psi (1+\mid d\psi \mid ^{2})^{\alpha}h(d\psi(e_{i}),  d\psi(e_{i}))
\nonumber \\&=\lambda\circ \psi \mid d\psi \mid ^{2}(1+\mid d\psi \mid ^{2})^{\alpha}, 
\end{align}
where we use the $\alpha-$harmonicity of $\psi$ for the second equality.  From \eqref{22r} and the divergence theorem,  we get 
\begin{align}\label{23r}
0&=\int _{M}div X \,  dV_{g}\nonumber \\&=\int _{M}\lambda\circ \psi \mid d\psi \mid ^{2}(1+\mid d\psi \mid ^{2})^{\alpha} dV_{g}.
\end{align}
This implies that $\mid d\psi \mid=0$.  Then,  $\psi$ is constant and hence completes the proof. 
\end{proof}

\section{Stability of $\alpha-$harmonic maps}
In this section,  the instability of non-constant $\alpha$-harmonic mappings with regard to the criterion of the Ricci curvature of the target manifold is investigated.  Then,  the instability of totally geodesic immersion mappings from a hypersurface $M^{n-1}$ to a Riemannian manifold $N^{n}$ with positive Ricci curvature is explored. \\

  By  Jacobi operator and Green's theorem,   the second variation formula of the $\alpha-$energy functional can be obtained as follows.
\begin{theorem}(The second variation formula)\label{12sd}
Let 
$ \psi :(M,g)\longrightarrow (N,h) $ be an $\alpha-$ harmonic map and 
$\{\psi_{t,s}:M\longrightarrow N \}_{-\epsilon <s, t<\epsilon}$
be a 2-parameter smooth variation of $ \psi $
such that
$ \psi_{0,0}=\psi. $
 Then
 \begin{align}\label{30i}
&\dfrac{\partial ^{2}}{\partial t \partial s}E_{\alpha }(\psi)\mid _{t=s=0}=- \int _{M}h(J_{\alpha}(\upsilon), \omega),
 \end{align}
 where 
$$\upsilon= \dfrac{\partial \psi_{t,s}}{\partial t}\mid _{s=t=0}, \qquad \omega= \dfrac{\partial \psi_{t,s}}{\partial s}\mid _{s=t=0}, $$
 and  the Jacobi  operator
 $J_{\alpha}(\upsilon)\in \Gamma (\psi ^{-1}TN)$ is given by 
 \begin{align}\label{36}
&J_{\alpha}(\upsilon)\nonumber \\&=2\alpha(1+\mid d\psi \mid ^{2})^{\alpha-1}Tr_{g} R^{N}(\upsilon, d\psi)d\psi 
\nonumber \\&+4\alpha(\alpha -1) Tr_{g} \nabla \langle \nabla \upsilon, d\psi\rangle (1+\mid d\psi \mid ^{2})^{\alpha-2}d\psi
\nonumber \\&+2\alpha Tr_{g}\nabla (1+\mid d\psi \mid ^{2})^{\alpha-1}\nabla \upsilon,
 \end{align}
 here $R^{N}$ is the curvature tensor on $(N,h)$ and $\langle ,\rangle$ denotes the inner product on $T^{\ast}M\otimes \psi ^{-1}TN. $
\end{theorem}
\begin{remark}
The Jacobi operator is used to calculate the nonlinear boundary value of Troesch's equation. This equation appears in the study of plasma column confinement by radiation pressure as well as in the theory of gas porous electrodes, \cite{m1}.
\end{remark}
\begin{definition}\label{35}
Under the assumptions of theorem \ref{12sd}, 
setting
\begin{equation}
I_{\alpha,  \psi}(\upsilon,\omega)=\dfrac{\partial ^{2}}{\partial t \partial s}E_{\alpha }(\psi_{t,s})\mid _{s=t=0}.
\end{equation}
An  $\alpha$-harmonic map $\psi$ is called $\alpha$-stable or stable if  $I_{\alpha,  \psi}(\omega,\omega)$ is non-negative for every vector field  $\omega$ along $\psi$,  or,  equivalently,  the eigenvalues of the elliptic self-adjoint operator
$J_{\alpha}$
are all non-negative.  Otherwise,  it is said to be  $\alpha-$unstable or unstable.  

\end{definition}
It is valuable to note that the stability of $\psi$ is equivalent to the fact that the eigenvalues of
the elliptic self-adjoint operator,$J_{\alpha}(\psi)$, are all non-negative.  Moreover,  any  $\alpha-$stable map $\psi: M\longrightarrow N$ is a local (resp.  global) minimum of the $\alpha$-energy  functional,  $E_{\alpha}$,     within a homotopy class of $L_{1}^{p}(M,N)$ having the same trace on $\partial M$.
\begin{remark}
The stability of harmonic maps plays a key role in mathematical physics and mechanics,  \cite{con11}.
For instance,  the linearized Vlasov-Maxwell equations are used to investigate harmonic stability properties for planar wiggler free-electron laser\textit{(FEL)}.  It should be noted that the analysis is performed in the Compton regime for a tenuous electron beam travelling in the $z$ direction via a constant amplitude planar wiggler magnetic field $B^{0}=-B_{\omega}cos k_{0}z\hat{e}_{x}$, \cite{D1}.
\end{remark}
From Theorem \ref{12sd}, we obtain the following corollary. 
\begin{corollary}\label{45}
Assume that $N$ is a Riemannian manifold with non-positive Riemannian curvature.  Any $\alpha-$harmonic map $ \psi :(M,g)\longrightarrow (N,h) $ is stable.
\end{corollary}
\begin{proof}
Setting 
\begin{equation}
\delta(X)=2\alpha (1+\mid d\psi \mid ^{2})^{\alpha-1}h(\nabla _{X}\upsilon, \omega), 
\end{equation} 
and 
\begin{align}\label{ui}
 \qquad \eta (X)&=4\alpha(\alpha-1) (1+\mid d\psi \mid ^{2})^{\alpha-2}\nonumber \\& \qquad\langle \nabla\upsilon , d\psi \rangle h(d\psi(X),\omega),
\end{align}
for any $X\in \chi(M)$. Then, we have
\begin{align}\label{37}
&-h(2\alpha Tr_{g}\nabla (1+\mid d\psi \mid ^{2})^{\alpha-1}\nabla \upsilon,\omega)
\nonumber \\&=-\sum _{i=1}^{m}h(2\alpha\nabla_{e_{i}}(1+\mid d\psi \mid ^{2})^{\alpha-1}\nabla_{e_{i}}\upsilon,\omega)
\nonumber \\&=-\sum _{i=1}^{m}\{e_{i}(h(2\alpha(1+\mid d\psi \mid ^{2})^{\alpha-1}\nabla_{e_{i}}\upsilon,\omega))
\nonumber \\&+h(2\alpha(1+\mid d\psi \mid ^{2})^{\alpha-1}\nabla_{e_{i}}\upsilon,\nabla_{e_{i}}\omega)\}
\nonumber \\&=-div\delta+2\alpha(1+\mid d\psi \mid ^{2})^{\alpha-1}\langle \nabla \upsilon, \nabla \omega \rangle,
\end{align}
On the other hand, by \eqref{ui}, we get 
\begin{align}\label{38}
&-h( Tr_{g} \nabla \langle \nabla \upsilon, d\psi\rangle B d\psi, \omega)
\nonumber \\&=\sum _{i=1}^{m}\{-h( \nabla_{e_{i}} \langle \nabla \upsilon, d\psi\rangle  A d\psi(e_{i}), \omega)\}
\nonumber \\&=
\sum _{i=1}^{m}\{-e_{i}(h( \langle \nabla \upsilon, d\psi\rangle A d\psi(e_{i}), \omega))
\nonumber \\&+
h( \langle \nabla \upsilon, d\psi\rangle A d\psi(e_{i}), \nabla_{e_{i}}\omega)\}
\nonumber \\&=-div \,\eta+A\langle \nabla \upsilon, d\psi\rangle \langle \nabla \omega, d\psi\rangle, 
\end{align}
where 
$A:=4\alpha(\alpha -1) (1+\mid d\psi \mid ^{2})^{\alpha-2}. 
$
By substituting \eqref{37} and \eqref{38} in \eqref{36} and using Divergence theorem and Definition \ref{35}, we get 
\begin{align}\label{40}
&I{\alpha,  \psi}(\upsilon,\upsilon)
\nonumber \\&=\int_{M}\{4\alpha(\alpha -1)(1+\mid d\psi \mid ^{2})^{\alpha-2}\langle \nabla \upsilon, d\psi\rangle ^{2}
\nonumber \\&
-2\alpha(1+\mid d\psi \mid ^{2})^{\alpha-1}h(Tr_{g} R^{N}(\upsilon, d\psi)d\psi, \upsilon )
\nonumber \\&
+2\alpha(1+\mid d\psi \mid ^{2})^{\alpha-1}\mid \nabla \upsilon \mid ^{2}
\}dV_{g}.
\end{align}
By equations \eqref{40} and the assumptions, the corollary \ref{45} follows. 
\end{proof}

Now,  by the Nash embedding theorem,  consider $N^{n}$ as a submanifold of codimension 
$p$ in an Euclidean space 
$\mathbb{R}^{n+p}$,  where the codimension $p$ is arbitrary.  Denote the inner product on $\mathbb{R}^{n+p}$ by $\langle,  \rangle$,  and the Levi-Civita connections on $N^{n}$  and $\mathbb{R}^{n+p}$ by 
$\nabla ^{N}$ and $\nabla ^{\mathbb{R}}$,  respectively.  At $x\in N$,  any vector 
$W$ on $\mathbb{R}^{n+p}$
can be decomposed as follows 
\begin{equation}\label{1b}
W=W^{\top}+W^{\bot},
\end{equation}
where 
$W^{\top}$ is the tangent part to 
$N$
and 
$W^{\bot}$ is the  normal part to 
$N$. 
The second fundamental form of 
$N^{n}$ 
in 
 $\mathbb{R}^{n+1}$
 is denoted by 
 $B$
and defined as follows 
\begin{equation}\label{2b}
B(X,Y)=(\nabla ^{\mathbb{R}}_{X}Y)^{\bot},
\end{equation}
where $X, Y$ are tangent vectors of 
$N^{n}$.
 Furthermore,  the shape operator 
$A^{V}$ corresponding to a normal vector field 
$V$
on $N^{n}$ is defined by 
\begin{equation}\label{3b}
A^{V}(X)=-(\nabla ^{\mathbb{R}}_{X}V)^{\top},
\end{equation}
where $X$ is a tangent vector field on 
$N^{n}$. 
Noting that,  the second fundamental form and the shape operator of $N^{n}$ are satisfied by the following equation
\begin{align}\label{4b}
\langle B(X,Y), V\rangle &=\langle A^{V}(X),    Y\rangle,
\end{align}
where  $X$ and $Y$  are tangent to  $N^{n}$ and  $V$ is normal to 
$N^{n}$  at $x$.  
 Define the function
$f: N\longrightarrow R$
by 
\begin{equation}\label{5b}
f(x)=max \{\mid B(\upsilon,\upsilon)\mid ^{2}:\upsilon\in T_{x} N, \, \mid \upsilon\mid=1 \},
\end{equation}
at each 
$x\in N$, and define the function 
$\theta :T_{x}N\longrightarrow \mathbb{R}$ by 
\begin{equation}\label{6b}
\theta(\upsilon)=\sum_{a=1}^{n}\mid B(\upsilon,\upsilon_{a})\mid ^{2},
\end{equation}
where $\{\upsilon_{a}\}_{a=1}^{n}$
is an orthonormal basis for 
$T_{x}N.$ It can be shown that the value of $\theta$ is independent of the choice of orthonormal basis.  Denote the the Ricci curvature of $N$ at $x$ in the direction $\upsilon$ by $Ric(\upsilon,\upsilon).$\\

Based on the preceding notations,  the instability of non-constant $\alpha$-harmonic maps with regard to the Ricci curvature of the target manifold is investigated in the following. 

\begin{theorem}
Let $\psi:(M^{m},g)\longrightarrow (N^{n},h)$
be a non-constant $\alpha-$harmonic map between Riemannian manifolds.  Suppose that 
\begin{equation}
2(\alpha-1)h+\theta(\upsilon)<Ric(\upsilon,\upsilon),
\end{equation}
at each $x\in N$ and any unit vector $\upsilon\in T_{x}N$.  Then,  $\psi$ is unstable.
\end{theorem}
\begin{proof}
Let $\{e_{i}\}_{i=1}^{m}$ be a local orthonormal frame field on $M$ and $\omega$ be a parallel vector field in $\mathbf{R}^{n+p}$.  Then,  by \eqref{40},  we have 
\begin{align}\label{7b}
&I{\alpha,  \psi}(\omega ^{\top},\omega ^{\top})
\nonumber \\&=\int_{M}\sum _{i=1}^{m}\{A_{1}\langle \nabla^{N} _{d\psi(e_{i})}\omega ^{\top}, d\psi(e_{i})\rangle ^{2}
\nonumber \\&
+A_{2}h( R^{N}(\omega ^{\top},d\psi(e_{i}) )\omega ^{\top}, d\psi(e_{i}) )
\nonumber \\&
+\mid \nabla _{d\psi(e_{i})}\omega ^{\top}\mid ^{2}
\}dV_{g}.
\end{align}
where $A_{1}:=4\alpha(\alpha -1)(1+\mid d\psi \mid ^{2})^{\alpha-2}$ and $A_{2}:=2\alpha(1+\mid d\psi \mid ^{2})^{\alpha-1}$.
By \eqref{1b} and \eqref{3b},  we get 
\begin{align}\label{8b}
\nabla _{d\psi(e_{i})} ^{N}\omega ^{\top}&=(\nabla _{d\psi(e_{i})}^{R} \omega ^{\top})^{\top}\nonumber \\&=(\nabla _{d\psi(e_{i})}^{R}( \omega- \omega ^{\bot}))^{\top}=-(\nabla _{d\psi(e_{i})}^{R} \omega ^{\bot})^{\top}
\nonumber \\&=
A^{ \omega ^{\bot}}(d\psi(e_{i})).
\end{align}
Then,  by \eqref{4b} and  \eqref{8b},  we get 
 \begin{align}\label{9b}
& \sum _{i=1}^{m}\langle \nabla^{N} _{d\psi(e_{i})}\omega ^{\top}, d\psi(e_{i})\rangle\nonumber \\&=
\sum _{i=1}^{m}\langle A^{ \omega ^{\bot}}(d\psi(e_{i})), d\psi(e_{i})\rangle \nonumber \\&=\sum _{i=1}^{m}\langle B(d\psi(e_{i}),  d\psi(e_{i})),  \omega ^{\bot}\rangle 
 \end{align}
By   \eqref{7b},  \eqref{8b} and  \eqref{9b},  it can be seen that 
\begin{align}\label{10b}
&I{\alpha,  \psi}(\omega ^{\top},\omega ^{\top})
\nonumber \\&=\int_{M}\sum _{i=1}^{m}\{A_{1}\sum _{i=1}^{m}\langle B(d\psi(e_{i}),  d\psi(e_{i})),  \omega ^{\bot}\rangle  ^{2}
\nonumber \\&
+A_{2}h( R^{N}(\omega ^{\top},d\psi(e_{i}) )\omega ^{\top}, d\psi(e_{i}) )
\nonumber \\&
+\mid A^{ \omega ^{\bot}}(d\psi(e_{i}))\mid ^{2}
\}dV_{g}.
\end{align}
Now,  by \eqref{10b},    we want to calculate the trace of $I{\alpha,  \psi}$.  Due to the fact that this trace is independent of the choice of an orthonormal basis for $ \mathbb{R} ^{n+p}$,  at any point $x\in N$.  We choose an orthonormal basis $\{\vartheta _{a},\vartheta _{\alpha}\}$,  $a=1,\cdots,  n$,  $\alpha=n+1,\cdots,  n+p$ such that $\{\vartheta _{a}\}_{a=1}^{n}$ are tangent to $N$ and $\{\vartheta _{\alpha}\}_{\alpha=n+1}^{n+p}$ are normal to $N$. Then,  we have 
\begin{align}\label{11b}
&\sum _{\alpha=n+1}^{n+p}\mid A^{ \vartheta _{\alpha} }(d\psi(e_{i}))\mid ^{2}\nonumber \\
&=\sum _{\alpha=n+1}^{n+p}\sum _{a=1}^{n}
\langle
 A^{ \vartheta _{\alpha} }(d\psi(e_{i})),  \vartheta _{a} \rangle  ^{2} 
\nonumber \\
&=
\sum _{a=1}^{n}\sum _{\alpha=n+1}^{n+p}
\langle
 B(d\psi(e_{i}),  \vartheta _{a} ),   \vartheta _{\alpha} \rangle  ^{2} 
 \nonumber \\
&=
\sum _{a=1}^{n}
\mid 
 B(d\psi(e_{i}),  \vartheta _{a} )\mid   ^{2}. 
\end{align}
Furthermore,  by Cauchy-Schwartz inequality,   we get 
\begin{align}\label{12b}
&\mid \sum_{i}
 B(d\psi(e_{i}),  d\psi(e_{i}) )\mid   ^{2}  \nonumber \\
 &= \sum_{i,j}\langle B(d\psi(e_{i}),  d\psi(e_{i}) ),  B(d\psi(e_{j}),  d\psi(e_{j}) ) \rangle
 \nonumber \\&\leq  \sum_{i,j}\mid B(d\psi(e_{i}),  d\psi(e_{i}) )\mid \mid B(d\psi(e_{j}),  d\psi(e_{j}) )\mid 
  \nonumber \\&\leq  \sum_{i,j}f(x)\mid d\psi(e_{i})\mid^{2}\mid d\psi(e_{j})\mid^{2}
   \nonumber \\&=\mid d\psi\mid^{2}\sum_{i}f(x)\mid d\psi(e_{i})\mid^{2}
\end{align}
Setting $\nu_{i}=d\psi(e_{i})/\mid d\psi(e_{i})\mid$,  for each $i$ such that $d\psi(e_{i})\neq 0$ at $x$.  Then,  by \eqref{10b}, \eqref{11b} and \eqref{12b},  we get 
\begin{align}\label{13b}
&Tr \, I{\alpha,  \psi}
\nonumber \\&=\int_{M}\sum _{i=1}^{m}\{A_{1}\mid \sum_{i}
 B(d\psi(e_{i}),  d\psi(e_{i}) )\mid   ^{2} 
\nonumber \\&
+A_{2}\sum_{a,i}\{ 
\mid  B(d\psi(e_{i}),  d\psi(e_{i}) )\mid   ^{2} 
\nonumber \\&+\langle R^{N}(\vartheta _{a},d\psi(e_{i}) )\vartheta _{a}, d\psi(e_{i}) \rangle
  \}
dV_{g}
\nonumber \\&\leq \int_{M}\sum _{i=1}^{m}\mid d\psi(e_{i})\mid^{2}\{
A_{1}\mid d\psi \mid ^{2}f(x)
\nonumber \\&+A_{2}(\theta (\nu_{i})-Ric(\nu_{i},\nu_{i}))
\}\nonumber \\&=2\alpha\int_{M}(1+\mid d\psi \mid^{2})^{\alpha-2}\sum _{i=1}^{m}\mid d\psi(e_{i})\mid^{2}\nonumber \\&\{
2(\alpha-1)\mid d\psi\mid^{2}f(x)+(1+\mid d\psi \mid^{2})(\theta(\nu_{i})\nonumber\\&-Ric(\nu_{i},\nu_{i})\}dV_{g}.
\end{align}
On the other hand,  by the assumptions, we get

\begin{align}\label{345}
&-2f(x)(\alpha-1)>2(\alpha-1)\mid d\psi\mid^{2}f(x)\nonumber\\&+(1+\mid d\psi \mid^{2})(\theta(\nu_{i})-Ric(\nu_{i},\nu_{i})
 \end{align} 
By \eqref{13b} \eqref{345},  it can be conclude that 
\begin{equation}
Tr \, I{\alpha,  \psi}<0,
\end{equation}
Thus,  $\psi$ is unstable.
\end{proof}
Now,  the instability of totally geodesic immersion maps from a hypersurface $M^{n-1}$ to a Riemannian manifold $N^{n}$ with positive Ricci curvature is studied. 
\begin{theorem}
Let $\psi:(M^{n-1},g)\longrightarrow (N^{n},h)$ be a totally geodesic isometric immersion of a hypersurface $M^{n-1}$ to $N^{n}$.  Then $\psi$ is unstable if the Ricci curvature of $N^{n}$ is positive. 
\end{theorem}
\begin{proof}
Choosing an orthonormal frame of $\{e_{i}\}_{i=1}^{n-1}$ on   $M^{n-1}$and  setting $\bar{e}_{i}:=d\psi(e_{i})$ for $i=1,\cdots ,  n-1$.  It can be seen that,   $\{\bar{e} _{i}\}_{i=1}^{n-1}$ forms an orthonormal basis on $M$.  Let $V$ be a unit normal vector field of $M^{n-1}$ in $N^{n}$.  Then,  $\{\bar{e} _{1}, \cdots \bar{e} _{n-1},  V \}$ is an orthonormal basis on $N$.  First,  we show that $\nabla^{N} _{\bar{e} _{i}}V=0$ for $i=1,\cdots,n-1$.  From $h(V,V)=0$,  it can be shown that
\begin{equation}\label{ef2}
h(\nabla^{N} _{\bar{e} _{i}}V,V)=0.  
\end{equation}
 Furthermore,  since  $V$ is normal to $M^{n-1}$,  and considering  \eqref{4b},    we get 
 \begin{align}\label{ef1}
  h(\nabla^{N} _{\bar{e} _{i}}V,\bar{e} _{j})&=  h((\nabla^{N} _{\bar{e} _{i}}V)^{\top},\bar{e} _{j})=-h(A^{V}(\bar{e} _{i}),\bar{e} _{j})\nonumber \\&=-h(B(\bar{e} _{i},\bar{e} _{j}),  V),
 \end{align}
where $(\nabla^{N} _{\bar{e} _{i}}V)^{\top}$ is the tangent part of $\nabla^{N} _{\bar{e} _{i}}V$ to $M^{n-1}$,  $A^{V}$ is the shape operator corresponding to $V$ and 
$B$ is the second fundamental form of $M^{n-1}$ in $N^{n}$.  Due to the fact that 
$M^{n-1}$ is totally geodesic,  then the second fundamental form $B$ vanishes.  Thus,  by \eqref{ef2} and \eqref{ef1},  it can be seen that 
\begin{equation}\label{ef3}
\nabla^{N} _{\bar{e} _{i}}V =0,   
\end{equation}
for   $i=1,\cdots,n-1. $     By \eqref{40} and \eqref{ef3},  we get 
\begin{equation}\label{ef4}
I_{\alpha,  \psi}(V,V)=-2\alpha n^{\alpha-1}\int _{M}Ric(V,V) dV_g.
\end{equation}
By  \eqref{ef4} and the assumptions,   it can be seen that $I_{\alpha,  \psi}(V,V)<0$ and hence completes the proof.
\end{proof}
\section{$\alpha$-stability of smooth manifolds}

 In this section,  the notion of $\alpha-$stable manifolds and its physical applications are studied.  Then,  
the instability of compact Riemannian manifolds  admitting a non-isometric conformal vector field is investigated.   Furthermore,   it is shown that on stable manifolds,  the assertion $I_{\alpha,id}(\upsilon,\upsilon)=0 $ is equivalent to the fact that the vector field 
$\upsilon$ is killing.  Finally,  the stability of Einstein Riemannian manifold $(M,g)$,  by considering an assumption on the smallest eigenvalue of the Laplacian operator on the functions of $M$,   is studied. 
\begin{definition}\label{ert}
Let $(M,g)$ be a compact Riemannian manifold and $\alpha$ be a positive constant with a value greater than one.   The manifold $M$ is said to be an $\alpha-$stable or stable if the identity map 
$id: (M,g)\longrightarrow (M,g)$ is  $\alpha-$stable.  Otherwise it is called unstable.
\end{definition}
\begin{remark}
According to physics, the concept of stable manifolds is a fundamental premise in
integration theory,  the stability of motion and  the inverse problem of Birkhoffian mechanics. 
 It is worth noting that Birkhoffian mechanics is an extension of Hamiltonian mechanics and is applicable to hadron physics and molecular physics, \cite{21e}. 
\end{remark}
\begin{theorem}\label{k3}
Let 
$(M^{m},g)$
be a compact Riemannian manifold and 
$\alpha$ be a positive constant with a value greater than one.  Then,  $M$ is an $\alpha$-stable manifold if $m^{2}-m<2\alpha$.  
\end{theorem}
\begin{proof}
Let 
$\{e_{i}\}$
be an orthonormal frame field on $M$.  By \eqref{40} and Definition  \ref{ert},  $I_{\alpha}(\upsilon,\upsilon)$ for the identity map 
$id:(M,g)\longrightarrow (M,g)$ can be rewritten as follows
\begin{align}\label{41}
&I_{\alpha,  id}(\upsilon,\upsilon)
\nonumber \\&=4\alpha(\alpha -1)(1+m)^{\alpha-2}\int_{M} \sum_{i=1}^{m}(div\, \upsilon)^{2}dV_{g}
\nonumber \\&+2\alpha(1+m)^{\alpha-1}\int_{M} g(J_{2,id}(\upsilon),\upsilon)dV_{g}
\end{align} 
where $J_{2,id}(\upsilon)= \Delta \upsilon-Ric(\upsilon).$ Here $\Delta$ is the rough Laplacian on $M$  and $Ric $ is the Ricci tensor of $N$. 
Now,  by definition of divergence operator and Cauchy-Schwarz inequality,  we get 
\begin{align}\label{etu}
(div\,  \upsilon)^{2}&=(\sum _{i=1}^{m}g(\nabla _{e_{i}}\upsilon,  e_{i}))^{2}
\nonumber \\&=
(\sum _{i=1}^{m}g(\nabla _{e_{i}}\upsilon,  e_{i})g(e_{i},e_{i}))^{2}
\nonumber \\&\leq \sum _{i=1}^{m}g(\nabla _{e_{i}}\upsilon,  e_{i})^{2}\sum _{i=1}^{m}g(e_{i},e_{i})^{2}
\nonumber \\&=m\sum _{i=1}^{m}g(\nabla _{e_{i}}\upsilon,  e_{i}).
\end{align}
Using     \eqref{etu},  it can be seen that  
\begin{equation}\label{v4}
\sum _{i=1}^{m}g(\nabla _{e_{i}}\upsilon,  e_{i})\geq \dfrac{1}{m} (div \,  \upsilon)^{2}.
\end{equation}
By \eqref{v4} and considering the assumptions,  we get 
\begin{align}\label{k1}
&\int_{M} g(J_{2,id}(\upsilon),\upsilon)dV_{g}
\nonumber \\&=\int_{M}\{\dfrac{1}{2} \mid \mathcal{L}  _{\upsilon} g\mid ^{2}-(div \, \upsilon)^{2}\}dV_{g}
\nonumber \\&=\dfrac{1}{2}\sum_{i,j}\int_{M} (\mathcal{L}_{\upsilon}g(e_{i},  e_{j}))^{2}dV_{g}-\int_{M} (div \,\upsilon)^{2}\}dV_{g}
\nonumber \\&=
\dfrac{1}{2}\sum_{i,j}\int_{M} (g(\nabla _{e_{i}}\upsilon,  e_{j})+g(e_{i}, \nabla _{e_{j}}\upsilon ))^{2}dV_{g}\nonumber \\&-\int_{M} (div \, \upsilon)^{2}\}dV_{g}
\nonumber \\&\geq 
\dfrac{1}{2}\sum_{i=1}^{m}\int_{M} 4 g(\nabla _{e_{i}}\upsilon,  e_{i})^{2}dV_{g}-\int_{M} (div \,\upsilon)^{2}dV_{g}
\nonumber \\&\geq (\dfrac{2}{m}-1)\int_{M} (div \, \upsilon)^{2} dV_{g}, 
\end{align}
where we use the following formula of 
K.  Yano,  see \cite{Yano1952}
\begin{align}\label{31i}
&\int_{M} g(J_{2,id}(\upsilon),\upsilon)dV_{g}\nonumber \\&=\int_{M}\{\dfrac{1}{2} \mid \mathcal{L}\upsilon  g\mid ^{2}-(div \,  \upsilon)^{2}\}dV_{g},
\end{align}
 for the second equality.  
 By \eqref{41},  \eqref{v4} and  \eqref{k1},  we get 
\begin{align}\label{k2}
 I_{\alpha , id}(\upsilon,\upsilon)\geq H\int_{M}(div \,\upsilon)^{2}dV_{g}\geq 0,
 \end{align} 
where $H=2\alpha(1+m)^{\alpha-1}(2\alpha+m-m^{2})/m$.   By \eqref{k2} and the assumptions,  Theorem \ref{k3} follows.
\end{proof}
In the following proposition,  the instability of compact Riemannian manifolds which admits a non-isometric conformal vector field is studied. 
\begin{proposition}
Let $(M^{m},g)$ be a compact Riemannian manifold and 
 $\alpha$ be a positive constant with a value greater than two.   Assume that there exists a non-isometric conformal vector field 
 $W$ on $M$.  Then,  $M$ is unstable if $2\alpha<m-1$. 
\end{proposition}
\begin{proof}
Since $W$ is a non-isometric conformal vector field, we have 
\begin{equation}
div W\neq 0
\end{equation}
 and 
\begin{equation}\label{e8}
\dfrac{1}{2} \mid \mathcal{L}  _{W}g\mid ^{2}=\dfrac{2}{m}(div W)^{2}
\end{equation}
For more details see \cite{Yano1952}.  By \eqref{41},   \eqref{e8} and the first quality of \eqref{k1},  we get 
\begin{align}\label{e9}
&I_{\alpha, id }(W,W)\nonumber \\&=4\alpha (\alpha -1)(1+m)^{\alpha-2}\int _{M}(div W)^{2} dV_{g}
\nonumber \\&+2\alpha(1+m)^{\alpha -1}\int _{M} g(J_{2,id}(W),  W)dV_{g}
\nonumber \\&=4\alpha (\alpha -1)(1+m)^{\alpha-2}\int _{M}(div W)^{2} dV_{g}
\nonumber \\&+2\alpha(1+m)^{\alpha -1}\int _{M} \dfrac{1}{2} \mid \mathcal{L}  _{V}g\mid ^{2}- (div W)^{2}dV_{g}\nonumber \\&= C\int (2\alpha m+2-m-m^{2}) (div W)^{2} dV_{g}
\end{align}
where $C=2\alpha(1+m)^{\alpha-2}/m$.  
By \eqref{e9} and  the assumptions,  it can be seen that 
\begin{equation}
I_{\alpha, id }(W,W)<0.
\end{equation}
This completes the proof. 
\end{proof}
In the following,  under some assumptions,  it is shown that on $\alpha$- manifolds,  the assertion $I_{\alpha,id}(\upsilon,\upsilon)=0 $ is equivalent to the fact that the vector field 
$\upsilon$ is killing.
\begin{proposition}\label{rtgh}
Let 
$(M^{m},g)$
be a Riemannian manifold and $\alpha$ is a positive constant such that $2\alpha\geq m+3$. Then $M$ is stable.   For $\alpha= (m+3)/2 $,   a vector field $V$ is killing  if and only if $I_{\alpha,id}(V,V)=0 $. 
\end{proposition}
\begin{proof}
By  \eqref{41} and  \eqref{31i},  we get 
\begin{align}\label{34rt}
 &I_{\alpha,id}(V,V)\nonumber \\&=2\alpha (1+m)^{\alpha -2}(2\alpha -m-3)\int _{M}(div(V))^{2}dV_{g}
 \nonumber \\&+2\alpha (1+m)^{\alpha -2}\int _{M}  \dfrac{1}{2} \mid \mathcal{L} _{V}g\mid ^{2}dV_{g}
 \end{align} 
From \eqref{34rt},  Proposition \ref{rtgh} follows.

\end{proof}
\begin{definition}
A Riemannian manifold $(M^{m}, g), n \geq 3, $  is Einstein if its Ricci tensor $Ric$ of type $(0, 2)$ is of the form  $Ric=\lambda g$,   where $\lambda$ is a
constant.
\end{definition}
\begin{remark}
The concept of Einstein manifolds developed as a result of research into exact solutions of the Einstein field equations as well as concerns of quasi-umbilical hypersurfaces. The Robertson-Walker spacetimes, for example, are Einstein manifolds.  Einstein manifolds play an important role in Riemannian Geometry as well as in general theory of relativity.  In geometry,  Einstein manifolds form a natural subclass of various classes of Riemannian manifolds by a curvature condition imposed on their Ricci tensor.
\end{remark}
Now,  the stability of Einstein Riemannian manifold $(M,g)$,  by considering an assumption on the smallest eigenvalue of the Laplacian operating on the functions of $M$,   is studied.
\begin{theorem}
Let 
$(M,g)$
be an Einstein Riemannian manifold whose Ricci tensor equals $\lambda g$ for some constant $\lambda$.  Then,  $M$ is an $\alpha-$stable manifold if and only if $\lambda \leq \mu_{1}(m+2\alpha-1)/(m+1), $
where 
$\mu_{1}$ is the smallest positive eigenvalue of the Laplacian operating on functions. 
\end{theorem}
\begin{proof}
Let 
$\chi(M)$
be the space of all smooth vector fields on $M$.  By the Hodge-Kodaira theory,  we have 
\begin{align}
\chi(M)=&\{X\in \chi(M)\mid div(X)=0\}\nonumber \\ &\oplus \{grad \kappa\mid  \kappa\in C^{\infty}(M) \},
\end{align}
In other words,  the space of the gradients of the functions is the orthogonal complement to the kernel of divergence operator in the space of all vector fields with respect to the $L^{2}$ norm.  Noting that,  $J_{\alpha,id}$ stabilizes this decomposition since the Ricci curvature is a scalar multiple of the Identity.  Hence,  we may work on the vector field 
$X$ with $div X=0$ and the gradients,  $grad \kappa$; that is,  we have  
\begin{align}
&g( J_{\alpha, id}(X+grad \kappa),  X+grad \kappa)\nonumber \\&=g(J_{\alpha, id}(X),X)+g(J_{\alpha, id}(grad \kappa),grad \kappa). 
 \end{align} 
 By \eqref{30i} and \eqref{31i} and considering 
 $div X=0$, 
 we get 
 \begin{align} 
& \int_{M}g(J_{\alpha, id}(X), X)\nonumber\\ &=D\int _{M}(Tr_{g}\nabla \langle \nabla X,  id_{TM}\rangle id_{TM},  X)dV_{g}\nonumber \\&+E\int _{M}h(J_{2,id}(X),  X)dV_{g}
\nonumber \\&=
D\int _{M}
 grad (div X) dV_{g}
\nonumber \\&+E
 \int_{M}\{\dfrac{1}{2} \mid \mathcal{L}_{X}  g\mid ^{2}-(div \,  X)^{2}\}dV_{g} \nonumber \\ &=E \int_{M}\dfrac{1}{2} \mid \mathcal{L}_{X}  g\mid ^{2}dV_{g}\geq 0,
\end{align} 
where $J_{2,id}$ is defined in \eqref{41} and 

$$D:=4\alpha(\alpha -1)(1+m)^{\alpha-2}, \qquad E:=2\alpha(1+m)^{\alpha-1}.$$
Assume that $M$ is $\alpha-$stable 
and 
$\kappa _{1}$ 
is a smooth function  on $M$ such that 
$\Delta \kappa _{1}=\mu_{1}\kappa _{1}$.
 Then we have 
\begin{align}\label{er5}
& \int_{M}g(J_{\alpha, id}(grad  \kappa_{1}), grad  \kappa_{1})\nonumber\\ &=D\int _{M}g(grad (\Delta \kappa_{1}), grad \kappa_{1}) dV_{g}
\nonumber\\ &+E\int _{M}g( grad \Delta\kappa_{1}-\lambda grad \kappa_{1},  grad \kappa_{1}) dV_{g}
\nonumber\\ &= 2\alpha(1+m)^{\alpha-2}(\mu _{1}(m+2\alpha-1)-\lambda(1+m) )\nonumber\\ &\int _{M}\mid grad \kappa_{1} \mid ^{2}
dV_{g}.
\end{align}
By \eqref{er5},  and  
 $\alpha-$stablity of $M$,  
it can be concluded that 
\begin{equation}
 \lambda \leq \mu_{1}(\dfrac{m+2\alpha-1}{m+1}).
 \end{equation}

Conversely, 
suppose  that  $\lambda \leq \mu_{1}(m+2\alpha-1)/(m+1)$. 
Due to the fact that 
\begin{align}
&\int _{M}g(grad \Delta \kappa, grad \kappa)dV_{g}\nonumber\\ &\geq \mu_{1}\int _{M}g(grad  \kappa, grad \kappa)dV_{g},
\end{align}
for every $\kappa \in C^{\infty}(M)$,   and considering \eqref{30i} and \eqref{41},  we have 
\begin{align}\label{er4}
&\int_{M}g(J_{\alpha, id}(grad  \, \kappa), grad \,  \kappa)
\nonumber\\ &
\nonumber\\ &=
D\int _{M}g(grad (\Delta\,  \kappa)), grad\,  \kappa) dV_{g}
\nonumber\\ &+E\int _{M}g( grad \, \Delta\kappa\nonumber -\lambda\,  grad \, \kappa, grad \,  \kappa) dV_{g}
\nonumber\\ &\geq 2\alpha(1+m)^{\alpha-2}(\mu _{1}(m+2\alpha-1)-\lambda(1+m) )\nonumber\\ &\int _{M}\mid grad\,  \kappa\mid ^{2} dV_{g}.
\end{align}

Since $\lambda \leq \mu_{1}(m+2\alpha-1)/(m+1)$,   equation \eqref{er4} implies  that 
\begin{equation}
\int_{M}g(J_{\alpha, id}(grad  \kappa_{1}), grad  \kappa_{1})\geq 0,
\end{equation}
and hence completes the proof. 
\end{proof}

\section*{Conflicts of Interest}
The authors declare that there are no conflicts of interest regarding the publication of this paper.


\end{document}